\documentclass[11pt]{article}
\usepackage[english]{babel}
\usepackage{amssymb,amsmath,amsthm}
\newtheorem{theorem}{Theorem}[section]
\newtheorem{lemma}[theorem]{Lemma}
\newtheorem{proposition}[theorem]{Proposition}
\newtheorem{corollary}[theorem]{Corollary}
\newtheorem{question}[theorem]{Question}

{\theoremstyle{definition}\newtheorem{definition}[theorem]{Definition}}
{\theoremstyle{definition}\newtheorem{remark}[theorem]{Remark}}

\numberwithin{equation}{section}

\def\C{{\mathbb C}}
\def\N{{\mathbb N}}
\def\Z{{\mathbb Z}}
\def\R{{\mathbb R}}
\def\K{{\mathbb K}}
\def\F{{\mathcal F}}
\def\pp{{\mathcal P}}
\def\epsilon{\varepsilon}
\def\kappa{\varkappa}
\def\phi{\varphi}
\def\leq{\leqslant}
\def\geq{\geqslant}
\def\dim{\hbox{\tt dim}\,}
\def\ker{\hbox{\tt ker}\,}
\def\spann{\hbox{\tt span}\,}
\def\deg{\hbox{\tt deg}\,}
\def\ssub#1#2{#1_{{}_{{\scriptstyle #2}}}}
\def\lbs{LB$_{\displaystyle\rm s}$}
\def\lbss{LB$_{\displaystyle s}$}

\title{Hypercyclic operators on topological vector spaces}

\author{Stanislav Shkarin}

\date{}

\begin{document}

\maketitle

\begin{abstract} Bonet, Frerick, Peris and Wengenroth constructed
a hypercyclic operator on the locally convex direct sum of countably
many copies of the Banach space $\ell_1$. We extend this result. In
particular, we show that there is a hypercyclic operator on the
locally convex direct sum of a sequence $\{X_n\}_{n\in\N}$ of
Fr\'echet spaces if and only if each $X_n$ is separable and there
are infinitely many $n\in\N$ for which $X_n$ is infinite
dimensional. Moreover, we characterize inductive limits of sequences
of separable Banach spaces which support a hypercyclic operator.
\end{abstract}



\section{Introduction \label{s1}}\rm

Unless stated otherwise, all vector spaces in this article are over
the field $\K$, being either the field $\C$ of complex numbers or
the field $\R$ of real numbers and all topological spaces {\it are
assumed to be Hausdorff}. When we speak of a topological vector
space $X$, we mean a {\it non-trivial} topological vector space:
$X\neq\{0\}$. As usual, $\Z$ is the set of integers, $\Z_+$ is the
set of non-negative integers and $\N$ is the set of positive
integers. Symbol $L(X,Y)$ stands for the space of continuous linear
operators from a topological vector space $X$ to a topological
vector space $Y$. We write $L(X)$ instead of $L(X,X)$ and $X'$
instead of $L(X,\K)$. By a {\it quotient} of a topological vector
space $X$ we mean the space $X/Y$, where $Y$ is a closed linear
subspace of $X$. For a subset $A$ of a vector space $X$, $\spann(A)$
stands for the linear span of $A$. For brevity, we say {\it locally
convex space} for a locally convex topological vector space. We also
often say {\it 'operator'} instead of {\it 'continuous linear
operator'}. Recall also that if $X$ is a topological vector space,
then $A\subset X'$ is called {\it uniformly equicontinuous} if there
exists a neighborhood $U$ of zero in $X$ such that $|f(x)|\leq 1$
for any $x\in U$ and $f\in A$. A subset $B$ of a topological vector
space $X$ is called {\it bounded} if for any neighborhood $U$ of
zero in $X$, a scalar multiple of $U$ contains $B$. A subset $A$ of
a vector space is called {\it balanced} if $zx\in A$ whenever $x\in
A$, $z\in\K$ and $|z|\leq 1$. A subset $D$ of a topological vector
space $X$ is called a {\it disk} if $D$ is convex, balanced and
bounded. For a disk $D$, the space $X_D=\spann(D)$ is endowed with
the norm $p_D$, being the Minkowskii functional \cite{vald} of the
set $D$. Boundedness of $D$ implies that the norm topology of $X_D$
is stronger than the topology inherited from $X$. If the normed
space $X_D$ is complete, then $D$ is called a {\it Banach disk}. It
is well-known \cite{bonet} that a compact disk is a Banach disk.

\begin{definition}\label{ELL1}We say that a sequence $\{x_n\}_{n\in\Z_+}$
of elements of a locally convex space $X$ is an $\ell_1$-{\it
sequence} if $x_n\to 0$ and the the series $\sum\limits_{n=0}^\infty
a_nx_n$ converges in $X$ for each $a\in\ell_1$.
\end{definition}

The following lemma is a standard and elementary fact, see for
instance \cite{bonet}. See also \cite{shk_new} for its version for
general topological vector spaces.

\begin{lemma}\label{l11} Let $\{x_n\}_{n\in\Z_+}$ be an
$\ell_1$-sequence in a locally convex space $X$. Then the closed
balanced convex hull $D$ of $\{x_n\}_{n\in\Z_+}$ is a compact and
metrizable disk. Moreover, the Banach space $X_D$ is separable and
$\spann\{x_n:n\in \Z_+\}$ is dense in $X_D$.
\end{lemma}

We say that $\tau$ is a {\it locally convex topology} on a vector
space $X$ if $(X,\tau)$ is a locally convex space. We say that the
topology $\tau$ of a topological vector space $X$ {\it is weak} if
$\tau$ is exactly the weakest topology making each $f\in Y$
continuous for some linear space $Y$ of linear functionals on $X$
separating points of $X$. An $\F$-space is a complete metrizable
topological vector space. A locally convex $\F$-space is called a
{\it Fr\'echet} space. If $\{X_\alpha:\alpha\in A\}$ is a family of
locally convex spaces, then their (locally convex) {\it direct sum}
is the algebraic direct sum $X=\bigoplus\limits_{\alpha\in
A}X_\alpha$ of the vector spaces $X_\alpha$ endowed with the
strongest locally convex topology, which induces the original
topology on each $X_\alpha$. Let $\{X_n\}_{n\in\Z_+}$ be a sequence
of vector spaces such that each $X_n$ is a subspace of $X_{n+1}$ and
each $X_n$ carries its own locally convex topology $\tau_n$ such
that $\tau_n$ is (maybe non-strictly) stronger than
$\tau_{n+1}\bigr|_{X_n}$. Then the (locally convex) {\it inductive
limit} of the sequence $\{X_n\}$ is the space
$X=\bigcup\limits_{n=0}^\infty X_n$ endowed with the strongest
locally convex topology $\tau$ such that
$\tau\bigr|_{X_n}\subseteq\tau_n$ for each $n\in\Z_+$. In other
words, a convex set $U$ is a neighborhood of zero in $X$ if and only
if $U\cap X_n$ is a neighborhood of zero in $X_n$ for each
$n\in\Z_+$. An {\it LB-space} is an inductive limit of a sequence of
Banach spaces. An {\it \lbss-space} is an inductive limit of a
sequence of separable Banach spaces. We use symbol $\phi$ to denote
the locally convex direct sum of countably many copies of the
one-dimensional space $\K$, while $\omega$ is the product of
countably many copies of $\K$.

Let $X$ and $Y$ be topological spaces and $\F$ be a family of
continuous maps from $X$ to $Y$. An element $x\in X$ is called {\it
universal} for $\F$ if the orbit $\{Tx:T\in \F\}$ is dense in $Y$
and $\F$ is said to be {\it universal} if it has a universal
element. We say that $\{T_n:n\in\Z_+\}$ is {\it hereditarily
universal} if $\{T_n:n\in A\}$ is universal for each infinite subset
$A$ of $\Z_+$. If $T:X\to X$ is continuous, then $T$ is called {\it
transitive} if for any non-empty open subsets $U$ and $V$ of $X$,
there is $n\in\Z_+$ such that $T^n(U)\cap V\neq\varnothing$. The map
$T$ is said to be {\it mixing} if for any non-empty open subsets $U$
and $V$ of $X$, there is $n\in\Z_+$ such that $T^k(U)\cap
V\neq\varnothing$ for each $k\geq n$. Recall that a topological
space $X$ is called a {\it Baire space} if the intersection of
countably many dense open subsets of $X$ is dense in $X$. According
to the classical Baire theorem, complete metric spaces are Baire.

Let $X$ be a topological vector space and $T\in L(X)$. A vector
$x\in X$ is called a {\it cyclic vector} for $T$ if
$\spann\{T^nx:n\in\Z_+\}$ is dense in $X$ and $T$ is called {\it
cyclic} if it has a cyclic vector. Recall also that $T$ is called
{\it multicyclic} if $\spann\{T^nx:n\in\Z_+,\ x\in A\}$ is dense in
$X$ for some finite set $A\subset X$. $T$ is called {\it
hypercyclic} if $\{T^n:n\in\Z_+\}$ is universal and a universal
element for this family is called a {\it hypercyclic vector for}
$T$. Similarly, $T$ is {\it supercyclic} if $\{zT^n:n\in\Z_+,\
z\in\K\}$ is universal and a universal element for this family is a
{\it supercyclic vector for} $T$. Finally, $T$ is called {\it
hereditarily hypercyclic} if $\{T^n:n\in\Z_+\}$ is hereditarily
universal. Study of linear operators from above classes goes under
the name 'chaotic linear dynamics'. We refer to the book
\cite{bama-book} and references therein. The following proposition
is a collection of well-known observations, see, for instance,
\cite{bama-book}.

\begin{proposition}\label{untr2} Any hypercyclic operator is transitive and
any hereditarily hypercyclic operator is mixing. If the underlying
space is Baire separable and metrizable, then the converse
implications hold$:$ any transitive operator is hypercyclic and any
mixing operator is hereditarily hypercyclic.
\end{proposition}

The question of existence of hypercyclic operators on various types
of topological vector spaces was intensely studied. Ansari
\cite{ansa1} and Bernal-Gonz\'ales \cite{bernal}, answering a
question of Herrero, showed independently that any separable
infinite dimensional Banach space supports a hypercyclic operator.
Using the same idea, Bonet and Peris \cite{bonper} proved that there
is a hypercyclic operator on any separable infinite dimensional
Fr\'echet space and demonstrated that there is a hypercyclic
operator on an inductive limit $X$ of a sequence $X_n$ for
$n\in\Z_+$ of separable Banach spaces provided there is $n\in\Z_+$
for which $X_n$ is dense in $X$. Grivaux \cite{gri} observed that
hypercyclic operators in \cite{ansa1,bernal,bonper} are mixing and
therefore hereditarily hypercyclic. The following proposition is a
version of a theorem in \cite{ansa1,bonper} and can be used to
formalize this observation. We provide a sketch of a proof for sake
of completeness.

\begin{proposition}\label{formal} Let $X$ be a locally convex
space, $\{x_n\}_{n\in\Z_+}$ be an $\ell_1$-sequence in $X$ with
dense linear span and $\{f_n\}_{n\in\Z_+}$ be a uniformly
equicontinuous sequence in $X'$ such that $f_n(x_m)=0$ whenever
$n\neq m$ and $f_n(x_n)\neq 0$ for each $n\in\Z_+$. Then the formula
$Tx=\sum\limits_{n=0}^\infty 2^{-n}f_{n+1}(x)x_n$ defines a
continuous linear operator on $X$ such that $I+T$ is hereditarily
hypercyclic.
\end{proposition}

\begin{proof} By Lemma~\ref{l11}, the closed balanced convex hull $D$ of
$\{x_n:n\in\Z_+\}$ is a compact disk in $X$, $X_D$ is a separable
Banach space and $E=\spann\{x_n:n\in\Z_+\}$ is dense in $X_D$.
Observe that the restriction $T_D$ of $T$ to $X_D$ belongs to
$L(X_D)$. By Theorem~2.2 in \cite{bama-book}, if $Y$ is a
topological vector space and $S\in L(Y)$ is such that the linear
span $\Lambda(S)$ of the union of $\ker S^n\cap S^n(Y)$ for $n\in\N$
is dense in $Y$, then $I+S$ is mixing. Since $E\subseteq
\Lambda(T_D)$, we can apply this result to $S=T_D$, to see that
$I+T_D$ is a mixing operator on $X_K$. By Proposition~\ref{untr2},
$I+T_D$ is hereditarily hypercylic. Since $X_D$ is dense in $X$ and
carries a topology stronger than the one inherited from $X$, $I+T$
is also hereditarily hypercyclic.
\end{proof}

Applying an inductive construction exactly as in the proof in
\cite{bonper} of existence of hypercyclic operators on separable
infinite dimensional Fr\'echet spaces, one can verify the following
proposition, see also \cite{shk_new} for a detailed proof in a more
general setting.

\begin{proposition}\label{for1} Let $X$ be a locally convex space
such that there exist an $\ell_1$-sequence in $X$ with dense linear
span and a linearly independent uniformly equicontinuous sequence in
$X'$. Then there are an $\ell_1$-sequence $\{x_n\}_{n\in\Z_+}$ in
$X$ with dense linear span and a uniformly equicontinuous sequence
$\{f_n\}_{n\in\Z_+}$ in $X'$ such that $f_n(x_m)=0$ whenever $n\neq
m$ and $f_n(x_n)\neq 0$ for each $n\in\Z_+$.
\end{proposition}

Let $X$ be a locally convex space. It is easy to see that the
following conditions are equivalent: (a) the topology of $X$ is not
weak, (b)there is a linearly independent uniformly equicontinuous
sequence $\{f_n\}_{n\in\Z_+}$ in $X'$ and (c) there is a continuous
seminorm $p$ on $X$ such that $\ker p$ has infinite codimension in
$X$. Thus, combining Propositions~\ref{formal} and \ref{for1}, we
arrive to the following result containing all the mentioned
existence theorems, except for the space $\omega$, which is a
separate case in \cite{bonper} anyway.

\begin{theorem}\label{for2} Let $X$ be a locally convex space such
that the topology of $X$ is not weak and there is an
$\ell_1$-sequence in $X$ with dense linear span. Then there exists a
hereditarily hypercyclic $T\in L(X)$.
\end{theorem}

\subsection{Results}

The simplest space with no $\ell_1$-sequences with dense linear span
is $\phi$. Bonet and Peris \cite{bonper} observed that there are no
supercyclic operators on $\phi$. On the other hand, Bonet, Frerick,
Peris and Wengenroth \cite{fre} constructed a hypercyclic operator
on the locally convex direct sum $X$ of countably many copies of
$\ell_1$. Clearly, $X$ is a complete \lbs-space with no
$\ell_1$-sequences with dense linear span. We characterize the
\lbs-spaces, supporting a hypercyclic operator.

\begin{theorem}\label{lbs} Let $X$ be the inductive limit of a
sequence $\{X_n\}_{n\in\Z_+}$ of separable Banach spaces. Then the
following conditions are equivalent:
\begin{itemize}\itemsep=-2pt
\item[{\rm (\ref{lbs}.1)}]$X$ supports no hypercyclic operators$;$
\item[{\rm (\ref{lbs}.2)}]$X$ supports no cyclic operators with dense
range$;$
\item[{\rm (\ref{lbs}.3)}]$X$ is isomorphic to $Y\times\phi$, where $Y$
is the inductive limit of a sequence $\{Y_n\}_{n\in\Z_+}$ of
separable Banach spaces such that $Y_0$ is dense in $Y;$
\item[{\rm (\ref{lbs}.4)}]for any sufficiently large $n$, $\overline{X}_{n+1}/\overline{X}_{n}$
is finite dimensional and the set $\{n\in\Z_+:\overline{X}_{n+1}\neq
\overline{X}_{n}\}$ is infinite, where $\overline{X}_{k}$ is the
closure of $X_k$ in $X$.
\end{itemize}
\end{theorem}

The proof is based upon the following result, which is of
independent interest.

\begin{theorem}\label{timesphi} Let $X$ be a topological vector
space, which has no quotients isomorphic to $\phi$. Then there are
no cyclic operators with dense range on $X\times \phi$.
\end{theorem}

The following two theorems provide more generalizations of the main
result of \cite{fre}.

\begin{theorem}\label{sumfre} There is a hypercyclic operator on the
direct sum of a sequence $\{X_n\}_{n\in\Z_+}$ of separable Fr\'echet
spaces if and only if $X_n$ are infinite dimensional for infinitely
many $n$.
\end{theorem}

\begin{theorem}\label{sumM} Let $X$ be the direct sum of a sequence
$\{X_n\}_{n\in\Z_+}$ of locally convex spaces such that the topology
of each $X_n$ is not weak and each $X_n$ admits an $\ell_1$-sequence
with dense linear span. Then there is a hypercyclic operator on $X$.
\end{theorem}

\section{$\ell_1$-sequences and equicontinuous sets}

\begin{lemma}\label{equi1}Let $Y_0$ and $Y_1$ be closed linear subspaces
of a locally convex space $Y$ such that $Y_0\subset Y_1$ and the
topology of $Y_1/Y_0$ is not weak. Then there is a uniformly
equicontinuous sequence $\{f_n:n\in\Z_+\}$ in $Y'$ such that
$\phi\subseteq \bigl\{\{f_n(y)\}_{n\in\Z_+}: y\in Y_1\bigr\}$ and
$f_n\bigr|_{Y_0}=0$ for each $n\in\Z_+$.
\end{lemma}

\begin{proof} Since the topology of $Y_1/Y_0$ is not weak, there
is a continuous seminorm $\widetilde p$ on $Y_1/Y_0$ such that $\ker
\widetilde p$ has infinite codimension in $Y_1/Y_0$. Clearly the
seminorm $p$ on $Y_1$ defined by the formula $p(y)=\widetilde
p(y+Y_0)$ is continuous, $\ker p$ has infinite codimension and
$Y_0\subseteq \ker p$. In particular $Y_p=Y_1/\ker p$ endowed with
the norm $\|x+\ker p\|=p(x)$ is an infinite dimensional normed
space. Since every infinite dimensional normed space admits a
biorthogonal sequence, we can choose sequences $\{y_n\}_{n\in\Z_+}$
in $Y_1$ and $\{g_n\}_{n\in\Z_+}$ in $Y'_p$ such that $\|g_n\|\leq
1$ and $g_n(y_k+\ker p)=\delta_{n,k}$ for any $n,k\in\Z_+$, where
$\delta_{n,k}$ is the Kronecker delta. Now let $h_n:Y_1\to \K$,
$h_n(y)=g_n(y+\ker p)$. The properties of $g_n$ can be rewritten in
terms of $h_n$ in the following way: $|h_n(y)|\leq p(y)$ and
$h_n(y_k)=\delta_{n,k}$ for any $n,k\in\Z_+$ and $y\in Y_1$. Since
any continuous seminorm on a subspace of a locally convex space
extends to a continuous seminorm on the entire space
\cite{shifer,bonet}, there is a continuous seminorm $q$ on $Y$ such
that $q\bigr|_{Y_1}=p$. By Hahn--Banach theorem, we can find $f_n\in
Y'$ such that $f_n\bigr|_{Y_1}=h_n$ and $|f_n(y)|\leq q(y)$ for any
$n\in\Z_+$ and $y\in Y$. Since $y_k\in Y_1$, the equalities
$f_n(y_k)=h_n(y_k)=\delta_{n,k}$ imply that $\phi\subseteq
\bigl\{\{f_n(y)\}_{n\in\Z_+}: y\in Y_1\bigr\}$. The inequalities
$|f_n(y)|\leq q(y)$ guarantee uniform equicontinuity of
$\{f_n:n\in\Z_+\}$. The same inequalities and the inclusion
$Y_0\subseteq \ker p\subseteq \ker q$ ensure that
$f_n\bigr|_{Y_0}=0$ for each $n\in\Z_+$.
\end{proof}

Applying Lemma~\ref{equi1} in the case $Y=Y_1$ and $Y_0=\{0\}$, we
get the following corollary.

\begin{corollary}\label{equi2}Let $Y$ be a locally convex space, whose
topology is not weak. Then there exists a uniformly equicontinuous
sequence $\{f_n:n\in\Z_+\}$ in $Y'$ such that $\phi\subseteq
\bigl\{\{f_n(y)\}_{n\in\Z_+}: y\in Y\bigr\}$.
\end{corollary}

\begin{lemma}\label{dera} Let $X$ and $Y$ be locally convex spaces such that
the topology of $X$ is not weak and $Y$ admits an $\ell_1$-sequence
$\{y_n\}_{n\in\Z_+}$ with dense linear span. Then there is $T\in
L(X,Y)$ such that $T(X)$ is dense in $Y$.
\end{lemma}

\begin{proof} By Corollary~\ref{equi2}, there is a
uniformly equicontinuous sequence $\{f_n\}_{n\in\Z_+}$ in $X'$ such
that $\phi\subseteq \bigl\{\{f_n(x)\}_{n\in\Z_+}: x\in X\bigr\}$.
Consider $T:X\to Y$ defined by the formula
$Tx=\sum\limits_{n=0}^\infty 2^{-n}f_n(x)y_n$, where the series
converges since $\{f_n(x)\}$ is bounded and $\{y_n\}$ is an
$\ell_1$-sequence. Since $\{f_n\}$ is uniformly equicontinuous,
there is a continuous seminorm $p$ on $X$ such that $|f_n(x)|\leq
p(x)$ for any $x\in X$ and $n\in\Z_+$. Since $\{y_n\}$ is an
$\ell_1$-sequence, Lemma~\ref{l11} implies that the closed convex
balanced hull $D$ of $\{y_n:n\in\Z_+\}$ is a disk. Since $q(y_n)\leq
1$, it is easy to see that $q(Tx)\leq 2p(x)$ for each $x\in X$,
where $q$ is the Minkowskii functional of $D$. Hence $T\in
L(X,Y_D)$. Since the topology of $Y_D$ is stronger than the one
inherited from $Y$, $T\in L(X,Y)$. The inclusion $\phi\subseteq
\bigl\{\{f_n(x)\}_{n\in\Z_+}: x\in X\bigr\}$ implies that $T(X)$
contains $\spann\{y_n:n\in\Z_+\}$, which is dense in $Y$. Thus $T$
has dense range.
\end{proof}

\begin{remark}\label{stable} Obviously if $X$ and $Y$ are
topological vector spaces and the topology of $X$ is not weak, then
the topology of $X\times Y$ is not weak. It is also easy to see that
the class of locally convex spaces admitting an $\ell_1$-sequence
with dense linear span contains separable Fr\'echet spaces and is
closed under finite or countable products.
\end{remark}

\section{Operators on $\phi\times X$ \label{P}}

Recall \cite{shifer} that $\phi$ can be interpreted as a linear
space of countable algebraic dimension carrying the strongest
locally convex topology. In this section we discuss certain
properties of $\phi$, mainly those related to continuous linear
operators. It is well known \cite{shifer} that $\phi$ is complete
and all linear subspaces of $\phi$ are closed. Moreover, infinite
dimensional subspaces of $\phi$ are isomorphic to $\phi$. It is also
well-known that for any topology $\theta$ on $\phi$ such that
$(\phi,\theta)$ is a topological vector space, $\theta$ is weaker
than the original topology of $\phi$. This observation implies the
following lemma.

\begin{lemma}\label{phi1} For any topological vector space $X$,
any linear map $T:\phi\to X$ is continuous.
\end{lemma}

\begin{lemma}\label{phi2} Let $X$ be a topological vector space and
$T:X\to\phi$ be a surjective continuous linear operator. Then $X$ is
isomorphic to $\phi\times\ker T$.
\end{lemma}

\begin{proof} Since $T$ is linear and surjective, there exists a
linear map $S:\phi\to X$ such that $TS=I$. By Lemma~\ref{phi1}, $S$
is continuous. Consider the linear maps $A:\phi\times\ker T\to X$
and $B:X\to \phi\times\ker T$ defined by the formulae $A(u,y)=y+Su$
and $Bx=(Tx,x-STx)$ respectively. It is easy to verify that $A$ and
$B$ are continuous, $AB=I$ and $BA=I$. Hence $B$ is a required
isomorphism.
\end{proof}

\begin{corollary}\label{phi3} Let $X$ be a topological vector space.
Then the following are equivalent$:$
\begin{itemize}\itemsep=-2pt
\item[{\rm (\ref{phi3}.1)}]$X$ is isomorphic to a space of the shape $Y\times\phi$,
where $Y$ is a topological vector space$;$
\item[{\rm (\ref{phi3}.2)}]$X$ has a quotient isomorphic to $\phi;$
\item[{\rm (\ref{phi3}.3)}]there is $T\in L(X,\phi)$ such that $T(X)$ is
infinite dimensional.
\end{itemize}
\end{corollary}

\begin{proof} The implications
$(\ref{phi3}.1)\Longrightarrow(\ref{phi3}.2)\Longrightarrow(\ref{phi3}.3)$
are trivial. Assume that there is $T\in L(X,\phi)$ with infinite
dimensional $T(X)$. Since any infinite dimensional linear subspace
of $\phi$ is isomorphic to $\phi$, $T(X)$ is isomorphic to $\phi$.
Hence there is a surjective $S\in L(X,\phi)$. By Lemma~\ref{phi2},
$X$ is isomorphic to $Y\times\phi$, where $Y=\ker S$. Hence
(\ref{phi3}.3) implies (\ref{phi3}.1).
\end{proof}

\subsection{Multicyclic operators on $\phi$}

Clearly $\phi$ is isomorphic to the space $\pp$ of all polynomials
over $\K$ endowed with the strongest locally convex topology. The
shift operator on $\phi$ is similar to the operator
\begin{equation}\label{M}
M:\pp\to\pp,\quad Mp(z)=zp(z).
\end{equation}

\begin{lemma}\label{cphi1} An operator $T\in L(\phi)$ is cyclic if
and only if $T$ is similar to $M$.
\end{lemma}

\begin{proof} Since $1$ is a cyclic vector for $M$, any
operator similar to $M$ is cyclic. Now let $T\in L(\phi)$ and
$x\in\phi$ be a cyclic vector for $T$. Then $T^nx$ for $n\in\Z_+$
are linearly independent. Indeed, otherwise their span is finite
dimensional, which contradicts cyclicity of $x$ for $T$. Since any
linear subspace of $\phi$ is closed, $\{T^nx:n\in\Z_+\}$ is an
algebraic basis of $\phi$. Then the linear map $J:\pp\to\phi$,
$Jp=p(T)x$ is invertible. By Lemma~\ref{phi1}, $J$ and $J^{-1}$ are
continuous. It is easy to see that $T=JMJ^{-1}$. Hence $T$ is
similar to $M$.
\end{proof}

\begin{lemma}\label{cphi2} Let $T\in L(\phi)$. Then the following
conditions are equivalent$:$
\begin{itemize}\itemsep=-2pt
\item[{\rm (\ref{cphi2}.1)}]$T$ is multicyclic$;$
\item[{\rm (\ref{cphi2}.2)}]there exist $k\in\N$ and a linear
subspace $Y$ of $\phi$ of finite codimension such that
$T(Y)\subseteq Y$ and the restriction $T\bigr|_Y\in L(Y)$ is similar
to $M^k$, where $M$ is defined in $(\ref{M})$.
\end{itemize}
\end{lemma}

\begin{proof} First, assume that (\ref{cphi2}.2) is satisfied. Pick
a finite dimensional subspace $Z$ of $\phi$ such that $\phi=Z\oplus
Y$. Since $T\bigr|_Y$ is similar to $M^k$, there is an invertible
linear operator $J:\pp\to Y$ for which $T\bigr|_Y=JM^kJ^{-1}$. Let
$L=Z+J(\pp_k)$, where $\pp_n=\{p\in\pp:\deg p<n\}$. Clearly $L$ is
finite dimensional. By the equality $T\bigr|_Y=JM^kJ^{-1}$,
$L+T(L)+{\dots}+T^{n-1}(L)\supseteq Z+J(\pp_{nk})$ for any $n\in\N$.
Hence the linear span of the union of $T^j(L)$ for $j\in\Z_+$
contains $Z+J(\pp)=Z+Y=\phi$. Since $L$ is finite dimensional, $T$
is multicyclic. That is, (\ref{cphi2}.2) implies (\ref{cphi2}.1).

Assume that $T$ is multicyclic. Then there is a subspace $L$ of
$\phi$ such that
\begin{equation}\label{multi}
\dim L=n\in\N\quad\text{and}\quad\spann\{T^kx:x\in L,\
k\in\Z_+\}=\phi,
\end{equation}
where we again use the fact that any linear subspace of $\phi$ is
closed and therefore any dense subspace of $\phi$ coincides with
$\phi$. We say that $x_1,\dots,x_m\in \phi$ are $T$-{\it
independent} if for any polynomials $p_1,\dots,p_m$, the equality
$p_1(T)x_1+{\dots}+p_n(T)x_m=0$ implies $p_1={\dots}=p_m=0$.
Otherwise, $x_1,\dots,x_m$ are $T$-{\it dependent}. Since
$T$-independence implies linear independence, any $T$-independent
subset of $L$ has at most $n$ elements. Let $A$ be a $T$-independent
subset of $L$ of maximal cardinality $k\leq n$. Since $A$ is
linearly independent, there is $B\subset L$ of cardinality $n-k$
such that $A\cup B$ is a basis in $L$. By definition of $k$, for any
$b\in B$, $A\cup\{b\}$ is $T$-dependent and therefore, using
$T$-independence of $A$, we can find polynomials $p_b$ and $p_{b,a}$
for $a\in A$ such that $p_b\neq 0$ and $p_b(T)b=\sum\limits_{a\in
A}p_{b,a}(T)a$. Let $m=0$ if $B=\varnothing$ and  $m=\max\{\deg
p_b:b\in B\}$ otherwise. Consider the spaces
$$
Z=\spann\{T^jb:b\in B,\ 0\leq j\leq m\}\ \ \text{and}\ \
Y=\spann\{T^ja:a\in A,\ j\in\Z_+\}.
$$
Then $Z$ is finite dimensional and $T(Y)\subseteq Y$. Obviously,
\begin{equation}\label{tind0}
T^ja\in Y\subseteq Y+Z\quad\text{for any $a\in A$ and $j\in\Z_+$}.
\end{equation}
Let $b\in B$ and $j\in\Z_+$. Since $p_b\neq 0$ and $\deg p_b\leq m$,
we can find polynomials $q,r$ such that $\deg r<m$ and
$t^j=q(t)p_b(t)+r(t)$. Then $T^jb=q(T)p_b(T)b+r(T)b$ and $r(T)b\in
Z$ since $\deg r<m$. Next, $p_b(T)b\in Y$ since
$p_b(T)b=\sum\limits_{a\in A}p_{b,a}(T)a$ and $q(T)p_b(T)b\in Y$
since $T(Y)\subseteq Y$. Thus
\begin{equation}\label{tind1}
T^jb\in Y+Z\quad\text{for any $b\in B$ and $j\in\Z_+$}.
\end{equation}
Since $A\cup B$ is a basis of $L$, (\ref{tind0}) and (\ref{tind1})
imply that $T^j(L)\subseteq Y+Z$ for each $j\in\Z_+$. By
(\ref{multi}), $\phi=Y+Z$. Since $Z$ is finite dimensional, $Y$ has
finite codimension in $\phi$. In particular, $Y$ is non-trivial and
$A\neq\varnothing$. That is, $1\leq k\leq n$ and
$A=\{a_1,\dots,a_k\}$. Now consider the linear operator $J:\pp\to
Y$, which sends the monomial $t^{l}$ to $T^{j}a_s$, where $j\in
\Z_+$ and $s\in \{1,\dots,k\}$ are uniquely defined by $l+1=jk+s$.
By definition of $Y$, $J$ is onto. By $T$-independence of $A$, $J$
is one-to-one. By definition of $J$, $Jt^{l+k}=TJt^l$. Hence $J$ is
invertible and $JM^k=T\bigr|_{Y}J$. That is, $M^k$ and $T\bigr|_{Y}$
are similar. Thus (\ref{cphi2}.1) implies (\ref{cphi2}.2).
\end{proof}

\begin{corollary}\label{dran} Let $T$ be a multicyclic operator on
$\phi$. Then $T$ is not onto.
\end{corollary}

\begin{proof} By Lemma~\ref{cphi2}, $\phi=Y\oplus Z$, where $Z$ has
finite dimension $m\in\Z_+$, $T(Y)\subseteq Y$ and $T\bigr|_Y$ is
similar to $M^k$ for some $k\in\N$. Since $T\bigr|_Y$ is similar to
$M^k$, $T^{m+1}(Y)$ has codimension $k(m+1)>m$ in $Y$. Hence $\dim
\phi/T^{m+1}(Y)>m$. On the other hand, $\dim T^{m+1}(Z)\leq \dim
Z=m$. Thus $T^{m+1}(\phi)=T^{m+1}(Z)+T^{m+1}(Y)$ has positive
codimension in $\phi$. Hence $T^{m+1}$ is not onto and so is $T$.
\end{proof}

\subsection{Proof of Theorem~\ref{timesphi}}

Let $X$ be a topological vector space with no quotients isomorphic
to $\phi$. We have to show that there are no cyclic operators with
dense range on $X\times \phi$. Assume the contrary and let $T\in
L(X\times\phi)$ be a cyclic operator with dense range. Clearly
$T(x,u)=(Ax+Bu,Cx+Du)$ for any $(x,u)\in X\times \phi$, where $A\in
L(X)$, $B\in L(\phi,X)$, $C\in L(X,\phi)$ and $D\in L(\phi)$. Since
$T$ is cyclic, we can pick a vector $(x,u)\in X\times \phi$ such
that $E=\spann\{T^k(x,u):k\in\Z_+\}$ is dense in $X\times\phi$.
Since $T$ has dense range, $T^m$ has dense range for any $m\in\Z_+$.
Thus $E_m=T^m(E)=\spann\{T^k(x,u):k\geq m\}$ is dense in
$X\times\phi$ for any $m\in\Z_+$. Let $T^k(x,u)=(x_k,u_k)$, where
$x_k\in X$ and $u_k\in\phi$. Since $E_m=\spann\{(x_k,u_k):k\geq m\}$
are dense in $X\times\phi$, $F_m=\spann\{u_k:k\geq m\}$ are dense in
$\phi$. Hence $F_m=\phi$ for any $m\in\Z_+$. Since $X$ has no
quotients isomorphic to $\phi$, Lemma~\ref{phi3} implies that
$L=\spann\bigl(C(X)\cup \{u\}\bigr)$ is a finite dimensional
subspace of $\phi$. Clearly $u_0=u\in L$ and
$u_{k+1}=Cx_{k}+Du_{k}\in Du_{k}+L$ for any $k\in\Z_+$. It follows
that each $u_{k}$ belongs to the space spanned by the union of
$D^m(L)$ for $m\in\Z_+$. Since $L$ is finite dimensional and the
linear span of all $u_k$ is $\phi$, $D$ is multicyclic. By
Lemma~\ref{cphi2}, we can decompose $\phi$ into a direct sum
$\phi=Y\oplus Z$, where $Z$ is finite dimensional, $D(Y)\subseteq Y$
and $D\bigr|_Y$ is similar to $M^n$ for some $n\in\N$. Then there is
an invertible $J\in L(Y,\pp)$ such that $D\bigr|_Y=J^{-1}M^nJ$. Let
also $P\in L(\phi)$ be the linear projection onto $Y$ along $Z$. We
consider two cases.

{\bf Case 1. } \ The sequence $\{\deg JPu_k\}$ is bounded. In this
case $\spann\{JPu_k:k\in\Z_+\}$ is finite dimensional. Since $JP$
has finite dimensional kernel, $F_0=\spann\{u_k:k\in\Z_+\}$ is
finite dimensional. We have arrived to a contradiction with the
equality $F_0=\phi$.

{\bf Case 2. } \ The sequence $\{\deg JPu_k\}$ is unbounded. Since
$N=(L+Z+D(Z))\cap Y$ is finite dimensional, $m=\sup\{\deg Jw:w\in
N\setminus\{0\}\}$ is finite: $m\in\Z_+$. We shall show that $\deg
JPu_{k+1}=n+\deg JPu_{k}$ whenever $\deg JPu_{k}>m$. Indeed, let
$k\in \Z_+$ be such that $\deg JPu_{k}>m$. By definition of $P$,
$u_k-Pu_k\in Z$ and $u_{k+1}-Pu_{k+1}\in Z$. As we know, $u_{k+1}\in
Du_{k}+L$. Hence $Pu_{k+1}\in DPu_k+L+Z+D(Z)$. Since $Pu_{k+1}$ and
$DPu_k$ belong to $Y$, we have $Pu_{k+1}\in DPu_k+N$. Thus there is
$w\in N$ such that $Pu_{k+1}=DPu_k+w$. Hence
$JPu_{k+1}=JDPu_k+Jw=M^nJPu_k+Jw$. Since $\deg M^nJPu_k=n+\deg
JPu_k>m\geq \deg Jw$, we have $\deg JPu_{k+1}=n+\deg JPu_k$. Since
$\{\deg JPu_k\}$ is unbounded, there is $k\in \Z_+$ such that $\deg
JPu_k>m$ and according to the just proven statement, we have $\deg
JPu_j=\deg JPu_k+n(j-k)$ for $j\geq k$. Since any family of
polynomials with pairwise different degrees is linearly independent,
$JPu_j$ for $j\geq k$ are linearly independent. Since $JP$ is a
linear operator, $u_j$ for $j\geq k$ are linearly independent. Hence
the sequence of spaces $\{F_j\}_{j\geq k}$ is strictly decreasing.
On the other hand, we know that $F_j=\phi$ for each $j\in\Z_+$. This
contradiction completes the proof.

\section{Hypercyclic operators on direct sums}

\begin{lemma}\label{aux1} Let $X$ and $Y$ be topological vector spaces
such that there exists $T\in L(X,Y\times\K)$ with dense range. Then
for any closed hyperplane $H$ of $X$, there exists $S\in L(H,Y)$
with dense range.
\end{lemma}

\begin{proof} We can express the restriction $T_0\in L(H,Y\times \K)$
of $T$ to $H$ as $T_0=(S_0,g)$, where $S_0\in L(H,Y)$ and $g\in H'$.
If $T_0$ has dense range, then $S=S_0$  is a required operator. It
remains to consider the case when the range of $T_0$ is not dense.
Since the range of $T$ is dense and $T_0$ is a restriction of $T$ to
a closed hyperplane, the codimension of $\overline{T_0(H)}$ in
$Y\times \K$ does not exceed 1. Hence this codimension is exactly
$1$ and there is a non-zero $\psi\in (Y\times\K)'$ such that
$\overline{T_0(H)}=\ker\psi$. If $\ker\psi=Y\times\{0\}$, then again
we can take $S=S_0$. If $\ker\psi\neq Y\times\{0\}$, there is $y\in
Y$ such that $\psi(y)=1$. It is straightforward to verify that $S\in
L(H,Y)$, $Sx=S_0x+g(x)y$ has dense range.
\end{proof}

\begin{lemma}\label{sumlcs} Let $\{X_n\}_{n\in\Z_+}$ be a sequence
of infinite dimensional locally convex spaces such that
\begin{itemize}\itemsep=-2pt
\item[{\rm(\ref{sumlcs}.1)}]there is a sequence $\{U_n\}_{n\in\Z_+}$
of subsets of $X_0$ such that $\spann(U_n)$ is dense in $X_0$ for
each $n\in\Z_+$ and for any non-empty open subset $U$ of $X_0$,
there is $m\in\Z_+$ for which $U_m\subseteq U;$
\item[{\rm(\ref{sumlcs}.2)}]there exists $T_0\in L(X_0,X_0\oplus
X_1)$ with dense range$;$
\item[{\rm(\ref{sumlcs}.3)}]for each $n\in\N$, there exists $T_n\in
L(X_n,X_{n+1}\times\K)$ with dense range.
\end{itemize}
Then there is a hypercyclic operator $S$ on
$X=\bigoplus\limits_{n=0}^\infty X_n$. \end{lemma}

\begin{proof}Let $Z_n=\{x\in X:x_j=0\ \ \text{for}\ \ j>n\}$ for
$n\in\Z_+$. Clearly $X$ is the union of the increasing sequence of
subspaces $Z_n$ and each $Z_n$ is naturally isomorphic to the direct
sum of $X_k$ for $0\leq k\leq n$. We shall construct inductively a
sequence of operators $S_k\in L(Z_k,Z_{k+1})$ and vectors $y_k\in
X_0$ satisfying the following conditions for any $k\in\Z_+$:
$$
\begin{array}{lll}
\text{(a1)\ \ $S_j=S_k\bigr|_{Z_j}$ for $ 0\leq j<k$;}& \text{(a3)\
\ $S_k\dots S_0y_k\notin Z_k$;}& \text{(a5)\ \ $y_k\in
U_k$.}\\
\text{(a2)\ \ $S_k(Z_{k})$ is dense in $Z_{k+1}$;}&\text{(a4)\ \
$S_k\dots S_0y_{k-1}=y_k$ if $k\geq 1$;}&
\end{array}
$$
By (\ref{sumlcs}.2), there is $S_0\in L(Z_0,Z_1)$ with dense range.
Since $Z_0$ is a proper closed subspace of $Z_1$ and $\spann(U_0)$
is dense in $X_0=Z_0$, we can pick $y_0\in U_0$ such that
$S_0y_0\notin Z_0$. The basis of induction has been constructed.
Assume that $n\in\N$ and $y_k\in X_0$, $S_k\in L(Z_k,Z_{k+1})$
satisfying (a1--a5) for $k<n$ are already constructed. By (a3) for
$k=n-1$, $w=S_{n-1}\dots S_0y_{n-1}\notin Z_{n-1}$. That is, the
$n^{\rm th}$ component $w_n$ of $w$ is non-zero. Since $X_n$ is
locally convex, we can pick a closed hyperplane $H$ in $X_n$ such
that $w_n\notin H$. Let $P\in L(Z_n)$ be the linear projection onto
$H$ along $Z_{n-1}\oplus \spann\{w_n\}$. By (\ref{sumlcs}.3) and
Lemma~\ref{aux1}, there is $R\in L(H,X_{n+1})$ with dense range.
According to (a3) for $k=n-1$, $S_{n-1}\dots S_0(Z_0)$ is dense in
$Z_n$. Hence $Q(Z_0)$ is dense in $X_{n+1}$, where $Q=RPS_{n-1}\dots
S_0$. Since $\spann(U_n)$ is dense in $Z_0$, we can pick $y_n\in
U_n$ such that $Qy_n\neq 0$. Since $Z_n=Z_{n-1}\oplus H\oplus
\spann\{w\}$, we define the linear map $S_n:Z_n\to Z_{n+1}$ by the
formula
$$
S_n(x+y+sw)=S_{n-1}x+Ry+sy_n\ \ \text{for}\ \ x\in Z_{n-1},\ y\in H\
\ \text{and}\ \ s\in\K.
$$
The operator $S_n$ is continuous since $S_{n-1}$ and $R$ are
continuous. Clearly (a1) and (a5) for $k=n$ are satisfied. Next,
$S_n(Z_n)\supseteq S_{n-1}(Z_{n-1})+R(H)$. By (a2) for $k=n-1$,
$S_{n-1}(Z_{n-1})$ is dense in $Z_{n}$. Since $R(H)$ is dense in
$X_{n+1}$, $S_n(Z_n)$ is dense in $Z_{n+1}=Z_n\oplus X_n$, which
gives us (a2) for $k=n$. Since $Qy_n\neq 0$, the last display
implies (a3) for $k=n$. Finally, since $S_nw=y_n$ from the
definition of $w$  we get (a4) for $k=n$. The inductive construction
of $S_k$ and $y_k$ is complete.

Condition (a1) ensures that there is a unique $S\in L(X)$ such that
$S\bigr|_{Z_n}=S_n$ for any $n\in\Z_+$. By (a4),
$S^{k+1}y_{k-1}=y_k$ for each $k\in\Z_+$ and therefore
$A=\{y_n:n\in\Z_+\}$ is contained in the orbit
$O=\{S^ny_0:n\in\Z_+\}$. By (a5) and (\ref{sumlcs}.1), $A$ is dense
in $X_0=Z_0$. By (a2), $S^m(A)$ is dense in $Z_m$ for each
$m\in\Z_+$. Since $A\subset O$, we have $S^m(A)\subset O$ and
therefore $O\cap Z_m$ is dense in $Z_m$ for any $m\in\Z_+$. Hence
$O$ is dense in $X$. That is, $y_0$ is a hypercyclic vector for $S$.
\end{proof}

\begin{remark}\label{sumsum} Condition (\ref{sumlcs}.1) is satisfied
if there exists a dense linear subspace $Y$ of $X_0$, carrying a
topology, stronger than the one inherited from $X_0$ and turning $Y$
into a separable metrizable topological vector space. Indeed any
countable base $\{U_n\}_{n\in\Z_+}$ of topology of $Y$ satisfies
(\ref{sumlcs}.1). In this case, the orbit $O$ in the proof of
Lemma~\ref{sumlcs} is not just dense. It is sequentially dense. The
latter property is strictly stronger than density already for
countable direct sums of separable infinite dimensional Banach
spaces.
\end{remark}

\begin{remark}\label{sumsum1} Lemma~\ref{sumlcs} remains true
(with virtually the same proof) if we replace locally convex direct
sum by the direct sum in the category of topological vector spaces.
In the latter case the condition of local convexity of $X_n$ can be
replaced by the weaker condition that $X'_n$ separates points of
$X_n$ for each $n\in\Z_+$.
\end{remark}

\subsection{Proof of Theorem~\ref{sumM}}

Let $X$ be the direct sum of a sequence $\{X_n\}_{n\in\Z_+}$ of
locally convex spaces such that the topology of each $X_n$ is not
weak and each $X_n$ admits an $\ell_1$-sequence with dense linear
span. By Lemma~\ref{l11}, there is a Banach disk $D$ in  $X_0$ such
that $(X_0)_D$ is separable and is dense in $X_0$. By
Remark~\ref{sumsum}, (\ref{sumlcs}.1) is satisfied. From
Remark~\ref{stable} and Lemma~\ref{dera} it follows that
(\ref{sumlcs}.2) and (\ref{sumlcs}.3) are also satisfied. By
Lemma~\ref{sumlcs}, there is a hypercyclic $T\in L(X)$.

\subsection{Proof of Theorem~\ref{sumfre}}

\begin{lemma}\label{aux2} Let $X$ and $Y$ be separable infinite
dimensional Fr\'echet spaces. Then there is no $T\in L(X,Y)$ with
dense range if and only if $X$ is isomorphic to $\omega$ and $Y$ is
not isomorphic to $\omega$.
\end{lemma}

\begin{proof} If both $X$ and $Y$ are isomorphic to $\omega$, then
there is a surjective $T\in L(X,Y)$. If $X$ is isomorphic to
$\omega$, $Y$ is not and $T\in L(X,Y)$, then $Z=T(X)$ carries
minimal locally convex topology \cite{bonet} since $\omega$ does. It
follows that $Z$ is either finite dimensional or isomorphic to
$\omega$ and therefore complete. Hence $Z$ is closed in $Y$ and
$Z=\overline{Z}\neq Y$ since $Y$ is neither finite dimensional nor
isomorphic to $\omega$. Thus there is no $T\in L(X,Y)$ with dense
range. It remains to show that there is $T\in L(X,Y)$ with dense
range if $X$ is not isomorphic to $\omega$. In the latter case the
topology of $X$ is not weak and it remains to apply Lemma~\ref{dera}
since any separable Fr\'echet space admits an $\ell_1$-sequence with
dense linear span.
\end{proof}

\begin{lemma}\label{cases} Let $X$ be the countable locally convex
direct sum of separable Fr\'echet spaces infinitely many of which
are infinite dimensional. Then $X$ is isomorphic to the locally
convex direct sum of a sequence $\{Y_n\}_{n\in\Z_+}$ of separable
infinite dimensional Fr\'echet spaces such that either $Y_n$ is
isomorphic to $\omega$ for each $n\geq 1$ or $Y_n$ is non-isomorphic
to $\omega$ for each $n\in\Z_+$.
\end{lemma}

\begin{proof} We know that $X$ is the direct sum of
$\{X_\alpha\}_{\alpha\in A}$, where $A$ is countable, each
$X_\alpha$ is a separable Fr\'echet space and $X_\alpha$ is infinite
dimensional for infinitely many $\alpha\in A$. If the set $B$ of
$\alpha\in A$ such that $X_\alpha$ is infinite dimensional and
non-isomorphic to $\omega$ is infinite, we can write
$A=\{\alpha_n:n\in\Z\}$, where $\alpha_n$ are pairwise different and
$\alpha_n\in B$ for each $n\in\Z_+$. Then $X$ is isomorphic to the
direct sum of $Y_n=X_{\alpha_n}\oplus X_{\alpha_{-n-1}}$ for
$n\in\Z_+$ and each $Y_n$ is a separable infinite dimensional
Fr\'echet space non-isomorphic to $\omega$.

If $B$ is finite, the set $C$ of $\alpha\in A$ for which $X_\alpha$
is isomorphic to $\omega$ is infinite. Hence we can write
$A\setminus B=\{\alpha_n:n\in\Z\}$, where $\alpha_n$ are pairwise
different and $\alpha_n\in C$ for each $n\in\Z_+$. Let
$Y_0=X_{\alpha_0}\oplus\bigoplus_{\alpha\in B}X_\alpha$ and
$Y_n=X_{\alpha_n}\oplus X_{\alpha_{-n}}$ for $n\in\N$. Since $B$ is
finite and $X_{\alpha_0}$ is isomorphic to $\omega$, $Y_0$ is a
separable infinite dimensional Fr\'echet space. Since for each
$n\in\N$, $X_{\alpha_n}$ is isomorphic to $\omega$ and
$X_{\alpha_{-n}}$ is either finite dimensional or isomorphic to
$\omega$, $Y_n$ is isomorphic to $\omega$ for any $n\in\N$. It
remains to  notice that $X$ is isomorphic to the direct sum of the
sequence $\{Y_n\}_{n\in\Z+}$.
\end{proof}

We are ready to prove Theorem~\ref{sumfre}. Let $X$ be a countable
infinite direct sum of separable Fr\'echet spaces. If all the spaces
in the sum, except for finitely many, are finite dimensional, then
$X$ is isomorphic to $Y\times \phi$, where $Y$ is a Fr\'echet space.
By Theorem~\ref{timesphi}, $X$ admits no cyclic operator with dense
range. In particular, there are no supercyclic operators on $X$. If
there are infinitely many infinite dimensional spaces in the sum
defining $X$, then according to Lemma~\ref{cases}, $X$ is isomorphic
to the locally convex direct sum of a sequence $\{Y_n\}_{n\in\Z_+}$
of separable infinite dimensional Fr\'echet spaces such that either
all $Y_n$ are non-isomorphic to $\omega$ or all $Y_n$ for $n\geq 1$
are isomorphic to $\omega$. In any case, by Lemma~\ref{aux2}, there
exists $T_0\in L(Y_0,Y_0\oplus Y_1)$ with dense range and there
exist $T_n\in L(Y_n,Y_{n+1}\times\K)$ with dense ranges for all
$n\in\N$. By Lemma~\ref{sumlcs} and Remark~\ref{sumsum}, there is a
hypercyclic operator on $X$. The proof of Theorem~\ref{sumfre} is
complete.

\section{Hypercyclic operators on countable unions of spaces}

\begin{lemma}\label{limit} Let a locally convex space $X$ be the union
of an increasing sequence $\{X_n\}_{n\in\N}$ of its closed linear
subspaces. Assume also that for any $n\in\N$ there is an
$\ell_1$-sequence with dense span in $X_n$ and the topology of
$X_n/X_{n-1}$ is not weak, where $X_0=\{0\}$. Then there exists a
linear map $S:X\to X$ and $x_0\in X_1$ such that $S\bigr|_{X_n}\in
L(X_n,X_{n+1})$ for any $n\in\N$ and $\{S^kx_0:k\in\Z_+\}$ is dense
in $X$.
\end{lemma}

Note that we do  not claim continuity of $S$ on $X$. Although if,
for instance, $X$ is the inductive limit of the sequence $\{X_n\}$,
then continuity of $S$ immediately follows from the continuity of
the restrictions $S\bigr|_{X_n}$.

\begin{proof}[Proof of $Lemma~\ref{limit}$] For $n\in\N$, let
$\{x_{n,k}\}_{k\in\Z_+}$ be an $\ell_1$-sequence with dense linear
span in $X_n$. For any $n\in\N$, we apply Lemma~\ref{equi1} with
$(Y,Y_1,Y_0)=(X,X_n,X_{n-1})$ to obtain a uniformly equicontinuous
sequence $\{f_{n,k}\}_{k\in\Z_+}$ in $X'$ such that each $f_{n,k}$
vanishes on $X_{n-1}$ and $\phi\subseteq
\bigl\{\{f_{n,k}(x)\}_{k\in\Z_+}:x\in X_{n}\bigr\}$. By
Lemma~\ref{l11}, there is a Banach disk $D$ in $X$ such that $X_D$
is a dense subspace of $X_1$ and the Banach space $X_D$ is
separable. Let $\{U_n\}_{n\in\N}$ be a base of topology of $X_D$. We
shall construct inductively a sequence of operators $S_k\in
L(X,X_{k+1})$ and vectors $y_k\in X_D$ satisfying the following
conditions for any $k\in\N$:
$$
\begin{array}{lll}
\text{(b1)\ \ $S_j\bigr|_{X_j}=S_k\bigr|_{X_j}$ for $1\leq j<k$;}&
\text{(b3)\ \ $f_{k+1,0}(S_k\dots S_1y_k)\neq0$;}& \text{(b5)\ \
$y_k\in U_k$.}
\\
\text{(b4)\ \ $S_k\dots S_1y_{k-1}=y_k$ if $k\geq 2$;}& \text{(b2)\
\ $S_k(X_{k})$ is dense in $X_{k+1}$;}&
\end{array}
$$
Consider the linear map $S_1:X \to X_2$ defined by the formula
$$
S_1 x=\sum_{k=0}^\infty 2^{-k}f_{1,k}(x) x_{2,k}.
$$
Since $\{x_{2,k}\}_{k\in\Z_+}$ is an $\ell_1$-sequence in $X_2$ and
$\{f_{1,k}:k\in\Z_+\}$ is uniformly equicontinuous, the above
display defines a continuous linear operator from $X$ to $X_2$.
Since $\phi\subseteq \bigl\{\{f_{1,k}(x)\}_{k\in\Z_+}:x\in
X_1\bigr\}$, $S_1(X_1)$ contains $\spann\{x_{2,k}:k\in\Z_+\}$. Hence
$S_1(X_1)$ is dense in $X_2$. Since $X_D$ is dense in $X_1$, $S_1$
has dense range and $X_2\cap \ker f_{2,0}$ is nowhere dense in
$X_2$, we can pick $y_1\in U_1$ such that $f_{2,0}(S_1y_1)\neq0$.
The basis of induction has been constructed. Assume now that
$n\geq2$ and $y_k\in X_D$, $S_k\in L(X,X_{k+1})$, satisfying
(b1--b5) for $k\leq n-1$, are already constructed. According to (b3)
for $k=n-1$, $f_{n,0}(w)\neq 0$, where $w=Ry_{n-1}$ and
$R=S_{n-1}\dots S_1$. Since $H=X_n\cap \ker f_{n,0}$ is a closed
hyperplane in $X_{n}$ and $w\notin H$, we have
$X_n=H\oplus\spann\{w\}$. Let $H_0=H\cap \ker f_{n,1}$. Then $H_0$
is a closed hyperplane of $H$. By (b2) for $k=n-1$, $R(X_1)$ is
dense in $X_n$. Since $X_D$ is dense in $X_1$, we can pick $y_n\in
U_n$ such that $u=Ry_n\notin H_0\oplus\spann\{w\}$. Thus
$X_{n}=H_0\oplus \spann\{u,w\}$. Pick any $v\in X_{n+1}$ such that
$f_{n+1,0}(v)\neq 0$ and let
$$
x_0=y_n-S_{n-1}w-\sum_{k=0}^\infty 2^{-k}f_{n,k+2}(w)x_{n+1,k}\ \
\text{and}\ \  x_1=v-S_{n-1}u-\sum_{k=0}^\infty
2^{-k}f_{n,k+2}(u)x_{n+1,k}.
$$
The above series converge since $\{x_{n+1,k}\}_{k\in\Z_+}$ is an
$\ell_1$-sequence and $\{f_{n,k}:k\in\Z_+\}$ is uniformly
equicontinuous. By construction of $H_0$, $u$ and $w$, the matrix
{\small $\begin{pmatrix}f_{n,0}(w)&f_{n,1}(w)\\
f_{n,0}(u)&f_{n,1}(u)\end{pmatrix}$} is invertible. This allows us
to find $y_0,y_1\in\spann\{x_0,x_1\}\subset X_{n+1}$ satisfying
$$
\text{$f_{n,0}(w)y_0+f_{n,1}(w)y_1=x_0$ and
$f_{n,0}(u)y_0+f_{n,1}(u)y_1=x_1$.}
$$
Consider the linear map $S_n:X\to X_{n+1}$ defined by the formula
$$
S_nx=S_{n-1}x+f_{n,0}(x)y_0+f_{n,1}(x)y_1+\sum_{k=0}^\infty
2^{-k}f_{n,k+2}(x)x_{n+1,k}.
$$
The above display defines a continuous linear operator since
$\{x_{n+1,k}\}_{k\in\Z_+}$ is an $\ell_1$-sequence and
$\{f_{n,k}:k\in\Z_+\}$ is uniformly equicontinuous. By the last
three displays, $S_nw=y_n$ and $S_nu=v$. From definition of $w$ and
$u$ and the relation $f_{n+1,0}(v)\neq 0$ it follows that (b3) and
(b4) for $k=n$ are satisfied. Clearly (b5) for $k=n$ is also
satisfied. Since each $f_{n,k}$ vanishes on $X_{n-1}$, we have from
the last display that $S_{n}x=S_{n-1}x$ for any $x\in X_{n-1}$.
Hence (b1) for $k=n$ is satisfied. It remains to verify (b2) for
$k=n$. Let $U$ be a non-empty open subset of $X_{n+1}$. Since
$E=\spann\{x_{n+1,k}:k\in\Z_+\}$ is dense in $X_{n+1}$, we can find
$x\in E$ and a convex balanced neighborhood $W$ of zero in $X_{n+1}$
such that $x+W \subseteq U$. Since $\phi\subseteq
\bigl\{\{f_{n,k}(x)\}_{k\in\Z_+}:x\in X_n\bigr\}$ and $x\in
E=\spann\{x_{n+1,k}:k\in\Z_+\}$, we can pick $y\in X_n$ such that
$f_{n,0}(y)=f_{n,1}(y)=0$ and $x=\sum\limits_{k=0}^\infty
2^{-k}f_{n,k+2}(y)x_{n+1,k}$. Hence $S_ny=S_{n-1}y+x$. By (b2) for
$k=n-1$, $S_{n-1}(X_{n-1})$ is dense in $X_n$. Since $S_{n-1}y\in
X_n$, we can find $r\in X_{n-1}$ such that $S_{n-1}r\in S_{n-1}y-W$.
By the already proven property (b1) for $k=n$, $S_{n-1}r=S_nr$.
Hence $S_nr\in S_{n-1}y-W$. Using the equality $S_ny=S_{n-1}y+x$, we
get $S_n(y-r)\in x+W\subseteq U$. Hence any non-empty open subset of
$X_{n+1}$ contains elements of $S_n(X_n)$, which proves (b2) for
$k=n$.  The inductive construction of $S_k$ and $y_k$ is complete.

By (b2), there is a unique linear map $S:X\to X$ such that
$S\bigr|_{X_n}=S_n\bigr|_{X_n}$ for any $n\in\N$. By (b4),
$S^{k+1}y_{k}=y_{k+1}$ for each $k\in\N$. Hence $A=\{y_n:n\in\N\}$
is contained in  $O=\{S^ny_1:n\in\Z_+\}$. By (b5), $A$ is dense in
$X_D$ and therefore is dense in $X_1$. By (b2), $S^m(A)$ is dense in
$X_{m+1}$ for each $m\in\Z_+$. Since $A\subset O$, we have
$S^m(A)\subset O$ and therefore $O\cap X_m$ is dense in $X_m$ for
each $m\in\N$. Hence $O$ is dense in $X$. Thus the required
condition is satisfied with $x_0=y_1$.
\end{proof}

Before proving Theorem~\ref{lbs}, we need to make the following two
elementary observations.

\begin{lemma}\label{lbs0}Let $X$ be an LB-space and $Y$ be a closed
linear subspace of $X$. Then either $X/Y$ is finite dimensional or
the topology of $X/Y$ is not weak.
\end{lemma}

\begin{proof} Since $X$ is an LB-space, it is the inductive limit of
a  sequence $\{X_n\}_{n\in\Z_+}$ of Banach spaces. If $X/Y$ is
infinite dimensional, we can find a linearly independent sequence
$\{f_n\}_{n\in\Z_+}$ in $X'$ such that each $f_n$ vanishes on $Y$.
Next, we pick a sequence $\{\epsilon_n\}_{n\in\Z_+}$ of positive
numbers converging to zero fast enough to ensure that
$\epsilon_n\bigl\|f_n\bigr|_{X_k}\bigr\|_k\to 0$ as $n\to\infty$ for
each $k\in\N$. It follows that $\epsilon_nf_n$ pointwise converge to
zero on $X$. Since any LB-space is barrelled \cite{shifer,bonet},
$\{\epsilon_n f_n:n\in\Z_+\}$ is uniformly equicontinuous. Hence
$p(x)=\sup\{\epsilon_n|f_n(x)|:n\in\Z_+\}$ is a continuous seminorm
on $X$. Since each $f_n$ vanishes on $Y$, $Y\subseteq \ker p$. Then
$\widetilde p(x+Y)=p(x)$ is a continuous seminorm on $X/Y$. Since
$f_n$ are linearly independent, $\ker p$ has infinite codimension in
$X$ and therefore $\ker\widetilde p$ has infinite codimension in
$X/Y$. Hence the topology of $X/Y$ is not weak.
\end{proof}

\begin{lemma}\label{lbs1} Let $X$ be an inductive limit of a
sequence $\{X_n\}_{n\in\Z_+}$ of Banach spaces such that $X_0$ is
dense in $X$. Then $X$ has no quotients isomorphic to $\phi$.
\end{lemma}

\begin{proof} Assume that $X$ has a quotient isomorphic to $\phi$.
By Lemma~\ref{phi2}, $X$ is isomorphic to $Y\times \phi$ for some
closed linear subspace $Y$ of $X$. Let $J:X_0\to X$ be the natural
embedding. Since $X_0$ is dense in $X$, $J$ has dense range. Hence
$J':X'\to X'_0$ is injective. Since $X$ is isomorphic to
$Y\times\phi$, $X'$ is isomorphic to $Y'\times \omega$ ($\omega$ is
naturally isomorphic to $\phi'$; here $X'$, $X'_0$, $Y'$ and $\phi'$
carry strong topology \cite{shifer,bonet}). Hence, there exists an
injective continuous linear operator from $\omega$ to the Banach
space $X'_0$. That is impossible, since any injective continuous
linear operator from $\omega$ to a locally convex space is an
isomorphism onto image and $\omega$ is non-normable.
\end{proof}

\subsection{Proof of Theorem~\ref{lbs}}

Throughout this section $X$ is the inductive limit of a sequence
$\{X_n\}_{n\in\Z_+}$ of separable Banach spaces. Let also
$\overline{X}_n$ be the closure of $X_n$ in $X$. First, we shall
prove that (\ref{lbs}.4) implies (\ref{lbs}.3). Assume that
(\ref{lbs}.4) is satisfied. Then we can pick a strictly increasing
sequence $\{n_k\}_{k\in\Z_+}$ of non-negative integers such that
$0<\dim \overline{X}_{n_{k+1}}/\overline{X}_{n_k}<\infty$ for each
$k\in\Z_+$. Hence, for any $k\in\Z_+$, there is a non-trivial finite
dimensional subspace $Y_k$ of $X_{n_{k+1}}$ such that
$\overline{X}_{n_k}\oplus Y_k=\overline{X}_{n_{k+1}}$. Thus the
vector space $X$ can be written as an algebraic direct sum
\begin{equation}\label{di1}
X=\overline{X_{n_0}}\oplus\bigoplus_{k=0}^\infty
Y_k=\bigcup\limits_{k=0}^\infty \overline{X_{n_0}}\oplus
Z_k,\quad\text{where}\quad Z_k=\bigoplus\limits_{j=0}^{k-1} Y_j.
\end{equation}
Apart from the original topology $\tau$ on $X$, we can consider the
topology $\theta$, turning the sum (\ref{di1}) into a locally convex
direct sum. Obviously $\tau\subseteq\theta$. On the other hand, if
$W$ is a balanced convex $\theta$-neighborhood of $0$ in $X$, then
$W\cap \overline{X_{n_k}}$ is a $\tau$-neighborhood of zero in
$\overline{X_{n_k}}$ for any $k\in\Z_+$. Indeed, it follows from the
fact that $\overline{X_{n_k}}=\overline{X_{n_0}}\oplus Z_k$, where
$Z_k$ is finite dimensional. Since the topology of each $X_{n_k}$ is
stronger than the one inherited from $X$, $W\cap X_{n_k}$ is a
neighborhood of zero in $X_{n_k}$ for each $k\in\Z_+$. Since $X$ is
the inductive limit of the sequence $\{X_{n_k}\}_{k\in\Z_+}$, $W$ is
a $\tau$-neighborhood of zero in $X$. Hence $\theta\subseteq\tau$.
Thus $\theta=\tau$ and therefore $X$ is isomorphic to
$\overline{X}_{n_0}\times Y$, where $Y$ is the locally convex direct
sum of $Y_k$ for $k\in\Z_+$. Since $Y_k$ are finite dimensional, $Y$
is isomorphic to $\phi$. Since $\overline{X}_{n_0}$ is the inductive
limit of the sequence $\{\overline{X}_{n_0}\cap
X_{n_k}\}_{k\in\Z_+}$ of separable Banach spaces (with the topology
inherited from $X_{n_k}$), the first one of which  is dense,
(\ref{lbs}.3) is satisfied. Thus (\ref{lbs}.4) implies
(\ref{lbs}.3).

Assume now that (\ref{lbs}.3) is satisfied. By Lemma~\ref{lbs1}, $Y$
has no quotients isomorphic to $\phi$. By Theorem~\ref{timesphi},
there are no cyclic operators with dense range on $X$. Thus
(\ref{lbs}.3) implies (\ref{lbs}.2). The implication
$(\ref{lbs}.2)\Longrightarrow(\ref{lbs}.1)$ is obvious since any
hypercyclic operator is cyclic and has dense range. It remains to
show that (\ref{lbs}.1) implies (\ref{lbs}.4). Assume the contrary.
That is, (\ref{lbs}.1) is satisfied and (\ref{lbs}.4) fails. The
latter means that either there is $n\in\Z_+$ such that
$\overline{X}_n$ is dense in $X$ or there is a strictly increasing
sequence $\{n_k\}_{k\in\Z_+}$ of non-negative integers such that
$\overline{X}_{n_{k+1}}/\overline{X}_{n_k}$ is infinite dimensional
for each $k\in\Z_+$. In the first case, $X$ supports an
$\ell_1$-sequence with dense linear span. By Lemma~\ref{lbs0}, the
topology of $X$ is not weak. By Theorem~\ref{for2}, there is a
hypercyclic operator on $X$, which contradicts (\ref{lbs}.1). It
remains to consider the case when there exists a strictly increasing
sequence $\{n_k\}_{k\in\Z_+}$ of non-negative integers such that
$\overline{X}_{n_{k+1}}/\overline{X}_{n_k}$ is infinite dimensional
for each $k\in\Z_+$. By Lemma~\ref{lbs0}, the topology of each
$\overline{X}_{n_{k+1}}/\overline{X}_{n_k}$ is not weak. Since each
$X_{n_k}$ is a separable Banach space, there is an $\ell_1$-sequence
$\{x_{k,j}\}_{j\in\Z_+}$ in $X_{n_k}$ with dense span. The same
sequence is an $\ell_1$-sequence with dense span in
$\overline{X}_{n_k}$. By Lemma~\ref{limit}, there is a linear map
$S:X\to X$ and $x_0\in X$ such that $\{S^kx_0:k\in\Z_+\}$ is dense
in $X$ and the restriction of $S$ to each $\overline{X}_{n_k}$ is
continuous. Since the topology of $X_{n_k}$ is stronger than the one
inherited from $X$, the restriction of $S$ to each $X_{n_k}$ is a
continuous linear operator from $X_{n_k}$ to $X$. Since $X$ is the
inductive limit of $\{X_{n_k}\}_{k\in\Z_+}$, $S:X\to X$ is
continuous. Hence $S$ is a hypercyclic continuous linear operator on
$X$, which contradicts (\ref{lbs}.1). The proof of the implication
$(\ref{lbs}.1)\Longrightarrow(\ref{lbs}.4)$ and that of
Theorem~\ref{lbs} is complete.

\section{Remarks on mixing versus hereditarily hypercyclic}

We start with the following remark. As we have already mentioned,
$\phi$ supports no supercyclic operators \cite{bonper}, which
follows also from Theorem~\ref{timesphi}. On the other hand, $\phi$
supports a transitive operator \cite{fre}. The latter statement can
be easily strengthened. Namely, take the backward shift $T$ on
$\phi$. That is, $Te_0=0$ and $Te_n=e_{n-1}$ for $n\geq 1$, where
$\{e_n\}_{n\in\Z_+}$ is the standard basis in $\phi$. By Theorem~2.2
from  \cite{bama-book}, $I+T$ is mixing. Thus we have the following
proposition.

\begin{proposition}\label{uuu} $\phi$ supports a mixing operator and supports
no supercyclic operators.
\end{proposition}

On the other hand, a topological vector space of countable algebraic
dimension can support a hypercyclic operator, as observed by several
authors, see \cite{fre}, for instance. The following proposition
formalizes and extends this observation.

\begin{proposition}\label{hmi}Let $X$ be a normed space of countable
algebraic dimension. Then there exists a hypercyclic mixing operator
$T\in L(X)$.
\end{proposition}

\begin{proof} The completion $\overline{X}$ of $X$ is a separable infinite
dimensional Banach space. By Theorem~\ref{for2}, there is a
hereditarily hypercyclic operator $S\in L(\overline{X})$. Let $x\in
\overline{X}$ be a hypercyclic vector for $S$ and
$E=\spann\{S^nx:n\in\Z_+\}$. Grivaux \cite{gri2} demonstrated that
for any two countably dimensional dense linear subspaces $E_1$ and
$E_2$ of a separable infinite dimensional Banach space $Y$, there is
an isomorphism $J:Y\to Y$ such that $J(E_1)=E_2$. Hence there is an
isomorphism $J: \overline{X}\to \overline{X}$ such that $J(X)=E$.
Let $T_0=J^{-1}SJ$. Since $J(X)=E$ and $E$ is $S$-invariant, $X$ is
$T_0$-invariant. Thus the restriction $T$ of $T_0$ to $X$ is a
continuous linear operator on $X$. Moreover, since the $S$-orbit of
$x$ is dense in $\overline{X}$, the $T_0$-orbit of $J^{-1}x$ is
dense in $\overline{X}$. Since $J^{-1}x\in X$, the latter orbit is
exactly the $T$-orbit of $J^{-1}x$ and $J^{-1}x$ is hypercyclic for
$T$. Hence $T$ is hypercyclic. Next, $T_0$ is mixing since it is
similar to the mixing operator $S$. Hence $T$ is mixing as a
restriction of a mixing operator to a dense invariant subspace.
\end{proof}

By Proposition~\ref{untr2}, if a topological vector space $X$ is
Baire separable and metrizable, then any mixing $T\in L(X)$ is
hereditarily hypercyclic. By Proposition~\ref{hmi}, there are mixing
operators on any countably dimensional normed space. The next
theorem implies that there are no hereditarily hypercyclic operators
on countably dimensional topological vector spaces, emphasizing the
necessity of the Baire condition in Proposition~\ref{untr2}.

\begin{theorem}\label{al1} Let $X$ be a topological vector space
such that there exists a hereditarily universal family
$\{T_n:n\in\Z_+\}\subset L(X)$. Then $\dim X>\aleph_0$.
\end{theorem}

\begin{proof} Since the topology of any topological vector space can
be defined by a family of quasinorms \cite{shifer}, we can pick a
non-zero continuous quasinorm $p$ on $X$. That is,
$p:X\to[0,\infty)$ is non-zero, continuous, $p(0)=0$, $p(x+y)\leq
p(x)+p(y)$ and $p(zx)\leq p(x)$ for any $x,y\in X$ and $z\in\K$ with
$|z|\leq 1$ and $(X,\tau_p)$ is a (not necessarily Hausdorff)
topological vector space, where $\tau_p$ is the topology defined by
the pseudometric $d(x,y)=p(x-y)$. The latter property implies that
$p(tx_n)\to 0$ for any $t\in\K$ and any sequence
$\{x_n\}_{n\in\Z_+}$ in $X$ such that $p(x_n)\to 0$.

Let $\kappa$ be the first uncountable ordinal (usually denoted
$\omega_1$). We construct sequences $\{x_\alpha\}_{\alpha<\kappa}$
and $\{A_\alpha\}_{\alpha<\kappa}$ of vectors in $X$ and subsets of
$\Z_+$ respectively such that for any $\alpha<\kappa$,
\begin{itemize}\itemsep=-2pt
\item[(s1)]$A_\alpha$ is infinite and $x_\alpha$ is a universal
vector for $\{T_n:n\in A_\alpha\}$;
\item[(s2)]$p(T_nx_\beta)\to 0$ as $n\to\infty$, $n\in A_\alpha$ for any
$\beta<\alpha$;
\item[(s3)]$A_\alpha\setminus A_\beta$ is finite for any
$\beta<\alpha$.
\end{itemize}

For the basis of induction we take $A_0=\Z_+$ and $x_0$ being a
universal vector for $\{T_n:n\in\Z_+\}$. It remains to describe the
induction step. Assume that $\gamma<\kappa$ and $x_\alpha$,
$A_\alpha$ satisfying (s1--s3) for $\alpha<\gamma$ are already
constructed. We have to construct $x_\gamma$ and $A_\gamma$
satisfying (s1--s3) for $\alpha=\gamma$.

{\bf Case 1:} $\gamma$ has the immediate predecessor. That is
$\gamma=\rho+1$ for some ordinal $\rho<\kappa$. Since $x_\rho$ is
universal for $\{T_n:n\in A_\rho\}$, we can pick an infinite subset
$A_\gamma\subset A_\rho$ such that $p(T_nx_\rho)\to 0$ as
$n\to\infty$, $n\in A_\gamma$. Since $A_\gamma$ is contained in
$A_\rho$, from (s3) for $\alpha\leq\rho$ it follows that
$A_\gamma\setminus A_\beta$ is finite for any $\beta<\gamma$. Hence
(s3) for $\alpha=\gamma$ is satisfied. Now from (s3) for
$\alpha=\gamma$ and (s2) for $\alpha<\gamma$ it follows that (s2) is
satisfied for $\alpha=\gamma$. Finally, since $\{T_n:n\in\Z_+\}$ is
hereditarily universal, we can pick $x_\gamma\in X$ universal for
$\{T_n:n\in A_\gamma\}$. Hence (s1) for $\alpha=\gamma$ is also
satisfied.

{\bf Case 2:} $\gamma$ is a limit ordinal. Since $\gamma$ is a
countable ordinal, there is a strictly increasing sequence
$\{\alpha_n\}_{n\in\Z_+}$ of ordinals such that
$\gamma=\sup\{\alpha_n:n\in\Z_+\}$. Now pick consecutively $n_0$
from $A_{\alpha_0}$, $n_1>n_0$ from $A_{\alpha_0}\cap A_{\alpha_1}$,
$n_2>n_1$ from $A_{\alpha_0}\cap A_{\alpha_1}\cap A_{\alpha_2}$ etc.
The choice is possible since by (s3) for $\alpha<\gamma$, each
$A_{\alpha_0}\cap {\dots} \cap A_{\alpha_n}$ is infinite. Now let
$A_\gamma=\{n_j:j\in\Z_+\}$. Since $A_\gamma\setminus
A_{\alpha_j}\subseteq\{n_0,\dots,n_{j-1}\}$, $A_\gamma\setminus
A_{\alpha_j}$ is finite for each $j\in\Z_+$. Now if $\beta<\gamma$,
we can pick $j\in\Z_+$ such that $\beta<\alpha_j<\gamma$. Then
$A_\gamma\setminus A_\beta\subseteq (A_\gamma\setminus
A_{\alpha_j})\cup (A_{\alpha_j}\setminus A_{\beta})$ is finite by
(s3) with $\alpha=\alpha_j$. Moreover, since $A_\gamma$ is contained
in $A_{\alpha_j}$ up to a finite set, from (s2) with
$\alpha=\alpha_j$ it follows that $p(T_nx_\beta)\to 0$ as
$n\to\infty$, $n\in A_\gamma$. Hence (s2) and (s3) for
$\alpha=\gamma$ are satisfied. Finally, since $\{T_n:n\in\Z_+\}$ is
hereditarily universal, we can pick $x_\gamma\in X$ universal for
$\{T_n:n\in A_\gamma\}$. Hence (s1) for $\alpha=\gamma$ is also
satisfied. This concludes the construction of $x_\alpha$ and
$A_\alpha$ satisfying (s1--s3) for $\alpha<\kappa$.

In order to prove that $\dim X>\aleph_0$, it suffices to show that
vectors $\{x_\alpha\}_{\alpha<\kappa}$ are linearly independent.
Assume the contrary. Then there are $n\in\N$,
$z_1,\dots,z_n\in\K\setminus\{0\}$ and ordinals
$\alpha_1<{\dots}<\alpha_n<\kappa$ such that
$z_1x_{\alpha_1}+{\dots}+z_n x_{\alpha_n}=0$. By (s2) with
$\alpha=\alpha_n$, $p(T_kx_{\alpha_j})\to 0$ as $k\to\infty$, $k\in
A_{\alpha_n}$ for $1\leq j<n$. Denoting $c_j=-z_j/z_n$ and using
linearity of $T_k$, we obtain
$T_kx_{\alpha_n}=c_1T_kx_{\alpha_1}+{\dots}+c_{n-1}T_kx_{\alpha_{n-1}}$
for any $k\in\Z_+$. Since $p$ is a quasinorm, we have
$$
p(T_k x_{\alpha_n})\leq \sum_{1\leq j<n}p(c_j T_kx_{\alpha_j})\to 0\
\ \text{as $k\to\infty$, $k\in A_{\alpha_n}$},
$$
which contradicts universality of $x_{\alpha_n}$ for $\{T_k:k\in
A_{\alpha_n}\}$  (=(s1) with $\alpha=\alpha_n$).
\end{proof}

\begin{corollary}\label{al2}A topological vector space of countable
algebraic dimension supports no hereditarily hypercyclic operators.
\end{corollary}

It is worth noting that there are infinite dimensional separable
normed spaces, which support no multicyclic or transitive operators.
We call a continuous linear operator $T$ on a topological vector
space $X$ {\it simple} if $T$ has shape $T=zI+S$, where $z\in\K$ and
$S$ has finite rank. It is easy to see that a simple operator on an
infinite dimensional topological vector space is never transitive or
multicyclic. We say that a topological vector space $X$ is {\it
simple} if it is infinite dimensional and any $T\in L(X)$ is simple.
Thus simple topological vector spaces support no multicyclic or
transitive operators. Examples of simple separable infinite
dimensional normed spaces can be found in
\cite{vald,pol,vmill,mar,shka,shka1}. Moreover, according to
Valdivia \cite{vald}, in any separable infinite dimensional Fr\'eche
spaced there is a dense simple hyperplane. All examples of this type
existing in the literature with one exception \cite{shka} are
constructed with the help of the axiom of choice and the spaces
produced are not Borel measurable in their completions. In
\cite{shka} there is a constructive example of a simple separable
infinite dimensional pre-Hilbert space $H$ which is a countable
union of compact sets. We end this section by observing that the
following question remains open.

\begin{question}\label{disu1}
Is there a hereditarily hypercyclic operator on a countable direct
sum of separable infinite dimensional Banach spaces?
\end{question}

The most of the above results rely upon the underlying space being
locally convex or at least having plenty of continuous linear
functionals and for a good reason. Recall that an infinite
dimensional topological vector space $X$ is called {\it rigid} if
$L(X)$ consists only of the operators of the form $\lambda I$ for
$\lambda\in\K$. Of course, $X'=\{0\}$ if $X$ is rigid. Clearly there
are no transitive or multicyclic operators on a rigid topological
vector space. Since there exist rigid separable $\F$-spaces
\cite{kal}, there are separable infinite dimensional $\F$-spaces
supporting no multicyclic or transitive operators. On the other
hand, the absence of non-zero continuous linear functionals on a
topological vector space does not guarantee the absence of
hypercyclic operators on it. It is well-known \cite{kal} that the
spaces $L_p[0,1]$ for $0\leq p<1$ are separable $\F$-spaces with no
non-zero continuous linear functionals. Consider $T\in L(L_p[0,1])$,
$Tf(x)=f(x/2)$. By \cite[Theorem~2.2]{bama-book}, $I+T$ is mixing
and therefore hereditarily hypercyclic. If the supply of continuous
linear functionals on a separable $\F$-space is large enough, the
existence of hypercyclic operators is guaranteed. Indeed, in
\cite{shk_new} it is shown that there are hereditary hypercyclic
operators on any separable $\F$-space $X$ with uncountable algebraic
dimension of $X'$. The case of finite positive dimension of $X'$
provides no challenge. Indeed, if $X$ is a topological vector space
with $0<n=\dim X'<\infty$, then $L=\{x\in X:f(x)=0\ \ \text{for each
$f\in X'$}\}$ is a closed linear subspace of $X$ of codimension $n$
invariant for each $T\in L(X)$. According to Wengenroth \cite{ww} a
hypercyclic operator has no closed invariant subspaces of finite
positive codimension, while a supercyclic operator has no closed
invariant subspaces of codimension $n\in\N$ with $n>1$ if $\K=\C$ or
$n>2$ if $\K=\R$. Thus the following proposition holds.

\begin{proposition}\label{wen} A topological vector space $X$
with $0<\dim X'<\infty$ supports no hypercyclic operators. If
additionally $\ssub{\dim}{\R}{X'}>2$, then there are no supercyclic
operators on $X$.
\end{proposition}

In particular, let $X=L_p[0,1]\times\K^n$ for $0\leq p<1$ and
$n\in\N$. Then $X$ supports no hypercyclic operators and $X$
supports no supercyclic operators if $\ssub{\dim}{\R}{\K^n}>2$. The
following problem remains unsolved.

\begin{question}\label{fspac}
Characterize $\F$-spaces supporting hypercyclic operators.
\end{question}

The same question can be asked about supercyclic or hereditarily
hypercyclic operators. We would like to emphasize the following
related question.

\begin{question}\label{fspac11}
Assume that an $\F$-space $X$ supports a hypercyclic operator. Is it
true that $X$ supports a hereditarily hypercyclic operator?
\end{question}

Note that an $\F$-space $X$ with countably dimensional dual can
support a hereditarily hypercyclic operator. For instance, one can
take $X=\omega$ of $X=L_p[0,1]\times\omega$ with $0\leq p<1$. For
all we have seen so far, we could have come up with the conjecture
that any separable $\F$-space with infinite dimensional dual
supports a hypercyclic operator. This conjecture turns out to be
false.

\begin{theorem}\label{exfr} There exists a separable $\F$-space
$X$ such that $X'$ is infinite dimensional and $X$ supports no
multicyclic operators.
\end{theorem}

We prove Theorem~\ref{exfr} with the help of the following
observation.

\begin{proposition}\label{exa1} Let $X$ be a rigid topological
vector space and $Y$ be a topological vector space such that any
operator in $L(Y,X)$ has finite rank and $L(X,Y)=\{0\}$. Then
$X\times Y$ supports no multicyclic operators.
\end{proposition}

\begin{proof} Let $T\in L(X\times Y)$. Using the assumptions on the
spaces $X$ and $Y$, we see that $T$ acts according to the formula
$T(x,y)=(\lambda x+By,Cy)$, where $\lambda\in\K$, $C\in L(Y)$, $B\in
L(Y,X)$ and $L=B(Y)$ is a finite dimensional subspace of $X$. It is
easy to see that $T^n(x,y)=(A_n(x,y),C^ny)$ for each $n\in\N$, where
$A_n\in L(X\times Y,X)$, $A_n(x,y)=\lambda^n x+B(C^{n-1}+\lambda
C^{n-2}+{\dots}+\lambda^{n-1}I)y$. Clearly $A_n(x,y)\in
\spann(L\cup\{x\})$ for any $x,y\in X\times Y$. It follows that for
any subspace $M$ of $X$, $A_n(M\times Y)\subseteq M+L$ for each
$n\in\N$. Since $M+L$ is finite dimensional if $M$ is finite
dimensional, we immediately see that $T$ is not multicyclic.
\end{proof}

\begin{corollary}\label{exa2} Let $X$ be a rigid topological vector
space with no subspaces isomorphic to $\omega$. Then $X\times
\omega$ supports no multicyclic operators.
\end{corollary}

\begin{proof} Since $X$ is rigid, $X'=\{0\}$. Then $g\circ S=0$ for
any $S\in L(X,\omega)$ and $g\in \omega'$. Since $\omega'$ separates
points of $\omega$, $S=0$. Hence $L(X,\omega)=\{0\}$. Next, let
$T\in L(\omega,X)$ be of infinite rank. Then $\omega/\ker T$ is
infinite dimensional. Since any infinite dimensional quotient of
$\omega$ is isomorphic to $\omega$, we factorizing out the space
$\ker T$, arrive to an injective $T_0\in L(\omega,X)$. Since
$\omega$ is a minimal topological vector space \cite{bonet}, $T_0$
is an isomorphism between $\omega$ and
$T_0(\omega)=T(\omega)\subseteq X$. That is, $X$ has a subspace
isomorphic to $\omega$, which is a contradiction. Thus any $T\in
L(\omega,X)$ has finite rank. It remains to apply
Proposition~\ref{exa1} to complete the proof.
\end{proof}

\begin{proof}[Proof of Theorem~\ref{exfr}]In \cite{kal}, one can
find examples of separable rigid $\F$-spaces $Z$ with no subspaces
isomorphic to $\omega$. For instance, there are rigid subspaces in
$L_p[0,1]$ for $0<p<1$. By Corollary~\ref{exa2}, $X=Z\times \omega$
supports no multicyclic operators. On the other hand, $X'$ is
infinite dimensional.
\end{proof}

\small\rm

\vskip1truecm

\scshape

\noindent Stanislav Shkarin

\noindent Queens's University Belfast

\noindent Department of Pure Mathematics

\noindent University road, Belfast, BT7 1NN, UK

\noindent E-mail address: \qquad {\tt s.shkarin@qub.ac.uk}


\begin{thebibliography}{99}

\itemsep=-2pt

\bibitem{ansa1}S.~Ansari, \it Existence of hypercyclic operators on
topological vector spaces, \rm J. Funct. Anal. \bf148\rm\ (1997),
384--390

\bibitem{bama-book}F.~Bayart and E.~Matheron, \it Dynamics of linear
operators, \rm Cambridge University Press, 2009

\bibitem{bernal}L.~Bernal-Gonz\'alez, \it On hypercyclic operators on
Banach spaces, \rm Proc. Amer. Math. Soc. \bf127\rm\ (1999),
1003--1010

\bibitem{bonet}J.~Bonet and P.~P\'erez-Carreras, \it
Barrelled locally convex spaces, \rm North-Holland Mathematics
Studies \bf131\rm, North-Holland Publishing Co., Amsterdam, 1987

\bibitem{bonper}J.~Bonet and A.~Peris, \it Hypercyclic operators on
non-normable Fr\'echet spaces, \rm J. Funct. Anal. \bf159\rm\
(1998), 587--595

\bibitem{fre}J.~Bonet, L.~Frerick, A.~Peris and J.~Wengenroth, \it
Transitive and hypercyclic operators on locally convex spaces, \rm
Bull. London Math. Soc. \bf37\rm\ (2005), 254--264

\bibitem{gri}S.~Grivaux, \it Hypercyclic operators, mixing operators
and the bounded steps problem, \rm J. Operator Theory, \bf54\ \rm
(2005), 147--168

\bibitem{gri2}S.~Grivaux, \it Construction of operators with prescribed
behaviour, \rm  Arch. Math. (Basel)  \bf81\rm\ (2003), 291--299

\bibitem{kal}N.~Kalton, N.~Peck and J.~Roberts, \it An $F$-space
sampler, \rm London Mathematical Society Lecture Note Series
\bf89\rm, Cambridge University Press, Cambridge, 1984

\bibitem{mar}W.~Marciszewski, \it An infinite dimensional pre-Hilbert
space without a continuous map on its own square, \rm Bull.  Acad.
Polon. Sci., ser. math. \bf31\rm\ (1983), 393--396

\bibitem{pol}R.~Pol, \it An infinite dimensional pre-Hilbert space
non-homeomorphic to its owm square, \rm Proc. Amer. Math. Soc.
\bf90\rm\ (1984), 450--454

\bibitem{shifer}H.~Sch\"afer, \it \newblock Topological Vector Spaces,
\rm MacMillan, New York, 1966

\bibitem{shka}S.~Shkarin, \it Three constructive examples of
pre-Hilbert spaces non-isomorphic to their closed hyperplanes, \rm
Russ. J. Math. Phys. \bf8\rm\ (2001), 106--121

\bibitem{shka1}S.~Shkarin, \it An infinite-dimensional pre-Hilbert space
that is non-homeomorphic to its closed hyperplane, \rm Izv. Math.
\bf63\rm\ (2000),  1263--1273

\bibitem{shk_new}S.~Shkarin, \it Hypercyclic and mixing semigroups of linear
operators \rm [preprint]

\bibitem{vald}M.~Valdivia, \it Topics in locally convex spaces, \rm
North Holland, Amsterdam, 1982

\bibitem{vmill}J.~Van Mill, \it An infinite-dimensional pre-Hilbert space all
bounded linear operators of which are simple, \rm Colloq. Math.
\bf54\rm\ (1987), 29--37

\bibitem{ww}J.~Wengenroth, \it Hypercyclic operators on non-locally
convex spaces, \rm Proc. Amer. Math. Soc. \bf131\rm\ (2003),
1759--1761

\end{thebibliography}
\end{document}